\documentclass[12pt,a4paper]{amsart}

\usepackage[T1]{fontenc}

 \usepackage[margin=1.09in]{geometry}

\usepackage{setspace}
\usepackage{graphicx}
\usepackage{xcolor}

\usepackage{hyperref}
\hypersetup{
  colorlinks=true,
  linkcolor=blue,
  citecolor=blue,
  urlcolor=blue
}
\providecommand{\doi}[1]{%
  \href{https://doi.org/#1}{\nolinkurl{https://doi.org/#1}}%
}

\numberwithin{equation}{section}

\theoremstyle{plain}
\newtheorem{thm}{Theorem}[section]
\newtheorem{prop}[thm]{Proposition}
\newtheorem{lemma}[thm]{Lemma}

\theoremstyle{definition}

\newtheorem{example}[thm]{Example}

\theoremstyle{remark}
\newtheorem{remark}[thm]{Remark}

\DeclareMathOperator{\Rm}{Rm}
\DeclareMathOperator{\Ric}{Ric}
\DeclareMathOperator{\Scal}{Scal}
\DeclareMathOperator{\grad}{grad}
\DeclareMathOperator{\Hess}{Hess}
\DeclareMathOperator{\tr}{tr}
\DeclareMathOperator{\Bach}{Bach}
\newcommand{\A}{\mathcal A}
\newcommand{\B}{\mathcal B}
\newcommand{\C}{\mathcal C}
\newcommand{\W}{\mathcal W}
\newcommand{\E}{\Ric_0}
\newcommand{\Ract}{\mathring R}
\newcommand{\Rzero}{\mathring R_0}
\newcommand{\vol}{dv_g}
\newcommand{\tf}[1]{(#1)_0}
\usepackage{enumitem}

\begin{document}

\title[complete critical metrics]{Critical metrics for the quadratic curvature functional
on complete four-manifolds}

\author[Y. Euh]{Yunhee Euh}
\address{Yunhee Euh\\ Department of Mathematics\\ Sungkyunkwan University\\ Suwon 16419, Korea}
\email{prettyfish@skku.edu}
%https://orcid.org/0000-0002-1425-3688}

\author[J.H. Park]{JeongHyeong Park}
\address{JeongHyeong Park \\ Department of Mathematics\\ Sungkyunkwan University\\ Suwon 16419, Korea}
\email{parkj@skku.edu}
%https://orcid.org/0000-0002-0336-3929}
\thanks{The research of YE was supported by the Basic Science Research Program through the National Research Foundation of Korea (NRF) funded by the Ministry of Education (No. RS-2019-NR040081). The research of JP was supported by the National Research Foundation of Korea (NRF) grant funded by the Korean government (MSIT) (RS-2024-00334956).}


\begin{abstract}
We study complete four-dimensional Riemannian metrics {on connected manifolds} that are critical for
the curvature functional $\A(g)=\int_M|\Rm|^2\,\vol$ and have finite
curvature energy. If
$\langle\Ract\E,\E\rangle\ge0$, where
$\E=\Ric-\frac14\Scal\,g$ and $\Ract$ is the curvature action on symmetric
two-tensors, then the metric is either Einstein or locally isometric to a
product of surfaces of constant Gaussian curvatures $c$ and $-c$, with
$c\ne0$. This extends the corresponding compact classification to the
complete setting with finite curvature energy.
\end{abstract}
\maketitle
%%%%%%%%%%%%%%%%%%%%%%%%%%%%%%%%%%%%%%%%%%%%%%%%%%%%%%%%%%%%%
\section{Introduction}
Throughout this paper, manifolds are assumed to be connected and without boundary.

Let $M$ be a smooth, oriented $n$-dimensional manifold, and let
$\mathfrak M=\mathfrak M(M)$ denote the space of all Riemannian metrics on
$M$. For $g\in\mathfrak M$, we consider the quadratic curvature functionals
\begin{equation*}
 \A(g)=\int_M|\Rm|^2\,\vol,\quad
 \B(g)=\int_M|\Ric|^2\,\vol,\quad
 \C(g)=\int_M\Scal^2\,\vol,
\end{equation*}
where $\Rm$, $\Ric$, and $\Scal$ denote the Riemann curvature tensor, the
Ricci tensor, and the scalar curvature, respectively, and $\vol$ is the
Riemannian volume form of $g$. On noncompact manifolds, criticality is understood
with respect to compactly supported variations. Quadratic curvature
functionals provide variational characterizations of distinguished
metrics; see \cite{Besse,SW,Vi2016}.

The study of these functionals has a long history.
Muto~\cite{Muto1974} studied critical metrics for $\A$
on compact orientable manifolds.
Besse~\cite{Besse} devotes the final part of Chapter~4 to quadratic
curvature functionals on compact manifolds. In particular, on a closed manifold, a critical metric for $\C$
is scalar flat when $n\ne4$, and is either scalar flat or Einstein
when $n=4$. Viaclovsky~\cite{Vi2016} surveyed critical metrics
for general quadratic curvature functionals, emphasizing the special
role of dimension four. In the complete setting, Kang~\cite{Kang2007}
classified complete, locally homogeneous four-dimensional metrics of finite
volume that are critical for the Riemann or Weyl curvature functional.
Catino~\cite{Catino} proved that a complete critical metric  for $\C$ with
positive scalar curvature has constant scalar curvature. In dimension
four, he also showed that a complete critical metric  for $\C$ with
nonnegative scalar curvature is either scalar flat or Einstein with
positive scalar curvature.

{{Dimension four is distinguished by the scale invariance of $\A$:
$\A(\lambda^2g)=\lambda^{n-4}\A(g)$ for $\lambda>0$. With full tensor
norms, the curvature decomposition in dimension four gives
\begin{equation*}
|\Rm|^2=|W|^2+2|\E|^2+\frac16\Scal^2,
 \qquad \E=\Ric-\frac14\Scal\,g,
\end{equation*}
where $W$ denotes the Weyl tensor. Thus finite curvature energy controls all
three terms. On a closed oriented four-manifold, the
Chern--Gauss--Bonnet formula also gives
\begin{equation}\label{eq:GB-introduction}
 \A=4\B-\C+32\pi^2\chi(M)
    =2\W+\frac13\C-32\pi^2\chi(M),
\end{equation}
where $ \W(g)=\int_M|W|^2\,\vol.$ The corresponding local variational identities remain valid for
compactly supported variations on an open manifold. The global
equalities in \eqref{eq:GB-introduction}, however, require a separate
analysis of the terms at infinity in the complete noncompact setting.
Examples of complete four-dimensional metrics with finite curvature
energy include Ricci-flat ALE (asymptotically locally Euclidean)
metrics and finite-volume locally symmetric metrics;
see also \cite{Vi2016} for background on curvature functionals.
In the locally symmetric case, finite curvature energy follows
from the constancy of $|\Rm|^2$ and the finite volume.
This class includes products $S^2(c)\times\Sigma(-c)$, where $c>0$
and $\Sigma(-c)$ is a complete surface of constant Gaussian
curvature $-c$ and finite area.
These products are weakly Einstein, meaning that
$R_i{}^{abc}R_{jabc}=\tfrac14|\Rm|^2g_{ij}$;
see \cite{EPS13} for the definition and the product example. These examples motivate the study of $\A$ on complete
four-dimensional manifolds with finite curvature energy,
without prescribing their volume or an asymptotic model.

In our earlier work with Sekigawa~\cite{DGAEPS2011}, we studied critical
metrics for $\A$, $\B$, and $\C$ on  four-dimensional manifolds.
Using a four-dimensional curvature identity arising from the
Chern--Gauss--Bonnet theorem~\cite{EPS13}, we obtained a classification
under a curvature condition involving the trace-free Ricci tensor.
We recall the relevant local result as Proposition~\ref{prop:DGAEPS}.

Here we consider complete four-dimensional metrics that are critical
for $\A$ and have finite curvature energy, that is, {\mbox{$\A(g)<\infty$}}.
Our curvature assumption is
\begin{equation}\label{eq:curvature-assumption}
 q:=\langle\Ract\E,\E\rangle
    =(\E)^{ab}R_{aibj}(\E)^{ij}\ge0,
 \qquad (\Ract h)_{ij}=R_{aibj}h^{ab}.
\end{equation}

The curvature convention is specified in Section~\ref{s:prel}.
The tensor $\E$ is the trace-free part of the Ricci tensor
and vanishes precisely when the metric is Einstein.
Condition~\eqref{eq:curvature-assumption} requires the curvature
quadratic form to be nonnegative on $\E$.
Its analytical role is reflected in the identity
\begin{equation*}
 \frac12\Delta|\E|^2
 =|\nabla\E|^2+2\langle\Ract\E,\E\rangle,
\end{equation*}
valid for a critical metric for $\A$ with constant scalar curvature.
Under \eqref{eq:curvature-assumption}, both terms on the
right-hand side are nonnegative.
Together with the $L^2$-control provided by finite curvature energy,
this identity allows standard cutoff estimates to show that
the Ricci tensor is parallel.
}}

The product example $S^2(c)\times\Sigma(-c)$ shows that
\eqref{eq:curvature-assumption} permits non-Einstein metrics
with mixed sectional curvature; see Section~\ref{s:prel}
for details.
{{
Example~\ref{rem:sign-counterexample} shows that this condition
cannot be omitted.
% The following theorem extends the compact classification
% of \cite{DGAEPS2011} to complete four-dimensional metrics
% with finite curvature energy satisfying
% \eqref{eq:curvature-assumption}.
}}

\begin{thm}\label{thm:main}
Let $(M,g)$ be a connected complete four-dimensional Riemannian manifold.
Suppose that $g$ is critical for $\A$, has finite curvature energy, and
satisfies $\langle\Ract\E,\E\rangle\ge0$ everywhere on $M$. Then $(M,g)$ is
either Einstein or locally isometric to a product of surfaces of
constant Gaussian curvatures $c$ and $-c$, where $c\ne0$.
\end{thm}

\begin{remark}
Theorem~\ref{thm:main} extends the compact classification in
\cite[Theorem~2.8]{DGAEPS2011} to complete metrics with finite
curvature energy, with the same local alternatives.
It imposes no sign restriction on the scalar curvature.
Catino's results for $\C$ recalled above assume positive or
nonnegative scalar curvature, whereas our result concerns $\A$
under finite curvature energy and the condition $q\ge0$.
\end{remark}
The  extension uses standard cutoff estimates on complete manifolds,
following the approach of Gaffney~\cite{Gaffney1954}
and Yau~\cite{Yau1976}.
Taking the trace of the Euler--Lagrange equation for $\A$
gives $\Delta\Scal=0$ in dimension four.
Finite curvature energy implies that $\Scal,\Ric\in L^2(M)$.
Standard cutoff estimates applied to the harmonic equation
$\Delta\Scal=0$ show that $\Scal$ is constant. 
Under \eqref{eq:curvature-assumption}, the Bochner identity
and a further cutoff estimate then imply that $\Ric$ is parallel. The local classification from
\cite{DGAEPS2011} completes the argument.

The paper is organized as follows. Section~\ref{s:prel} fixes the notation and curvature conventions,
presents the variational formulas and known classification results,
and examines the curvature assumption through examples,
including a counterexample showing that it cannot be omitted. Section~\ref{s.proof} proves Theorem~\ref{thm:main}. Section~\ref{s:lowdim} discusses
critical metrics for $\A$ in dimensions other than four.

%%%%%%%%%%%%%%%%%%%%%%%%%%%%%%%%%%%%%%%%%%%%%%%%%%%%%%%%%%%%%%
\section{Variational formulas and curvature conditions}\label{s:prel}

%%%%%%%%%%%%%%%%%%%%%%%%%%%%%%%%%%%%%%%%%%%%%%%%%%%%%%%%%%%%%%
\subsection{Notation and first variation formulas}
%%%%%%%%%%%%%%%%%%%%%%%%%%%%%%%%%%%%%%%%%%%%%%%%%%%%%%%%%%%%%%
Let $M$ be a connected, smooth, oriented manifold without boundary. We write
$\nabla$ for the Levi-Civita connection. Our curvature convention is
\begin{equation*}
 R(X,Y)Z=[\nabla_X,\nabla_Y]Z-\nabla_{[X,Y]}Z,
 \quad
 \Rm(X,Y,Z,U)=g(R(X,Y)Z,U).
\end{equation*}
Thus $R_{ijkl}$ are the components of $\Rm$,
$\Ric_{ij}=g^{ab}R_{aijb}$, and $\Scal=g^{ij}\Ric_{ij}$.
In an orthonormal frame, the sectional curvature of the plane spanned
by $e_a,e_b$ is $K_{ab}=R_{abba}$. All norms are full tensor norms.
Repeated upper and lower indices are summed. We use the
rough Laplacian $\Delta T=\nabla^a\nabla_aT$, so that
$\int_M f\Delta f\,\vol=-\int_M|\nabla f|^2\,\vol$ for compactly
supported $f$.

For a functional $\mathcal F$, its variational gradient is denoted by
$\grad\mathcal F$ and is characterized by
\begin{equation*}
 \left.\frac{d}{dt}\right|_{t=0}\mathcal F(g+th)
 =\int_M\langle\grad\mathcal F,h\rangle\,\vol,
 \qquad h\in C_c^\infty(S^2T^*M),
\end{equation*}
whenever the functional is finite. A metric is critical for
$\mathcal F$ if $\grad\mathcal F=0$. When the total functional is
infinite, the same equation is understood locally by varying the
energy on a relatively compact domain containing $\operatorname{supp}h$.
%In particular, a complete critical metric  for $\A$ has finite curvature
% energy if $\int_M|\Rm|^2\,\vol<\infty$.

Berger's formulas~\cite[Sections~5--6]{Beg} give
\begin{equation}\label{eq:EL-ABC}
\begin{aligned}
 (\grad\A)_{ij}
 &=-2R_i{}^{abc}R_{jabc}-4\Delta\Ric_{ij}
     +2\nabla_i\nabla_j\Scal\\
 &\quad+4\Ric_{ia}\Ric_j{}^a-4\Ric^{ab}R_{iabj}
     +\frac12|\Rm|^2g_{ij},\\[2mm]
 (\grad\B)_{ij}
 &=-2\Ric^{ab}R_{iabj}-\Delta\Ric_{ij}
     +\nabla_i\nabla_j\Scal\\
     &\quad-\frac12(\Delta\Scal)g_{ij}
     +\frac12|\Ric|^2g_{ij},\\[2mm]
 (\grad\C)_{ij}
 &=2\nabla_i\nabla_j\Scal-2(\Delta\Scal)g_{ij}
     -2\Scal\,\Ric_{ij}+\frac12\Scal^2g_{ij}.
\end{aligned}
\end{equation}
If $g$ is critical for $\C$, taking the trace of
the third equation in \eqref{eq:EL-ABC} gives
\begin{equation*}
 0=-2(n-1)\Delta\Scal
   +\left(\frac n2-2\right)\Scal^2.
\end{equation*}
Hence
\begin{equation*}
 \Delta\Scal=\frac{n-4}{4(n-1)}\Scal^2.
\end{equation*}
% Taking the trace of $\grad\C=0$ yields
% \begin{equation*}
%  -2(n-1)\Delta\Scal+\left(\frac n2-2\right)\Scal^2=0,
%  \qquad
%  \Delta\Scal=\frac{n-4}{4(n-1)}\Scal^2.
% \end{equation*}
This gives the following classical result.

\begin{prop}[{\cite[p.~133]{Besse}}]
Let $(M,g)$ be a closed, oriented $n$-dimensional Riemannian manifold.
If $g$ is critical for $\C$, then it is scalar flat when $n\ne4$, and
either scalar flat or Einstein when $n=4$.
\end{prop}

Catino~\cite{Catino} obtained the following results in the complete
setting.

\begin{prop}[{\cite[Theorem~1.1]{Catino}}]\label{Catino-Thm}
Let $(M,g)$ be a complete $n$-dimensional Riemannian manifold, $n\ge3$.
If $g$ is critical for $\C$ and has positive scalar curvature, then
its scalar curvature is constant.
\end{prop}

\begin{prop}[{\cite[Theorem~1.3]{Catino}}]
Let $(M,g)$ be a complete four-dimensional Riemannian manifold.
If $g$ is critical for $\C$ and has nonnegative scalar curvature,
then it is either scalar flat or Einstein with positive scalar curvature.
\end{prop}

If $g$ is a critical metric for $\A$, taking the trace
of the first equation in \eqref{eq:EL-ABC} gives
\begin{equation*}
 0=\left(\frac n2-2\right)|\Rm|^2-2\Delta\Scal.
\end{equation*}
Hence
\begin{equation*}\label{Lap}
 \Delta\Scal=\frac{n-4}{4}|\Rm|^2.
\end{equation*}
In dimension four this reduces to
\begin{equation}\label{eq:tau-harmonic}
 \Delta\Scal=0.
\end{equation}

%%%%%%%%%%%%%%
\subsection{Four-dimensional identities and classification results}
%%%%%%%%%%%%%%%%%
We next write the four-dimensional variational equation in invariant
form. For a symmetric two-tensor $h$, define
\begin{equation*}
 (\Ract h)_{ij}=R_{aibj}h^{ab},
 \qquad
 \Delta_L h=\Delta h-2\Ract h-\Ric\circ h-h\circ\Ric,
\end{equation*}
where $\Delta_L$ is the Lichnerowicz Laplacian and $\circ$ denotes
composition after identifying symmetric two-tensors with
endomorphisms using $g$, so that $
 (h\circ k)_{ij}=h_i{}^a k_{aj}.$
We write $h^2=h\circ h$ and, in dimension four,
$h_0=h-\frac14(\tr h)g$.
The local variational form of \eqref{eq:GB-introduction} gives
$\grad\A=4\grad\B-\grad\C$.
Substituting $\Ric=\E+\frac14\Scal\,g$ into
\eqref{eq:EL-ABC}, we obtain
\begin{equation}\label{eq:EL-rough}
 \grad\A
 =-4\Delta\E+8\Ract\E+2|\E|^2g
       +2\Hess\Scal-(\Delta\Scal)g.
\end{equation}
Using the definition of the Lichnerowicz Laplacian,
we can equivalently write this as
\begin{equation}\label{eq:EL-Lichnerowicz}
 \grad\A
 =-4\Delta_L\E-8\tf{\E^2}-2\Scal\,\E
       +2\tf{\Hess\Scal}-\frac12(\Delta\Scal)g.
\end{equation}
For the Weyl tensor with the same curvature convention, normalize the
Bach tensor by
\begin{equation*}
 (\Bach)_{ij}=\nabla^a\nabla^bW_{iabj}
                  +\frac12\Ric^{ab}W_{iabj},
 \qquad \grad\W=-4\Bach.
\end{equation*}
The contracted Bianchi identity and the Weyl decomposition give the
Weitzenb\"ock formula
\begin{equation*}
 \Bach=\frac12\Delta_L\E+\tf{\E^2}
          +\frac16\Scal\,\E-\frac16\tf{\Hess\Scal}.
\end{equation*}
Consequently \eqref{eq:EL-Lichnerowicz} is equivalently
\begin{equation*}\label{eq:EL-Bach}
 \grad\A=-8\Bach
  +\frac23\bigl(\Hess\Scal-(\Delta\Scal)g-\Scal\,\E\bigr).
\end{equation*}
These normalizations can be compared with
\cite[Proposition~5.3]{Vi2016}. Thus criticality for $\A$
is equivalent to
$\Delta\Scal=0$ and
$\Bach=\frac1{12}(\Hess\Scal-\Scal\,\E)$.
In particular, it coincides with Bach-flatness for scalar-flat metrics;
for constant nonzero scalar curvature, a Bach-flat metric is critical
for $\A$ precisely when it is Einstein.

For a critical metric $g$ for $\A$ on a four-dimensional manifold,
taking the inner product of \eqref{eq:EL-rough} with $\E$ and using
$\grad\A=0$ gives
\begin{equation*}
 \frac12\Delta|\E|^2
 =|\nabla\E|^2+2\langle\Ract\E,\E\rangle
  +\frac12\langle\Hess\Scal,\E\rangle.
\end{equation*}
Since $\Delta\Scal=0$ and
\begin{equation*}
 |\Ric|^2=|\E|^2+\frac14\Scal^2,
 \qquad
 |\nabla\Ric|^2
 =|\nabla\E|^2+\frac14|\nabla\Scal|^2,
\end{equation*}
we obtain the following identity.

\begin{lemma}[{\cite[p.~645]{DGAEPS2011}}]\label{lem:dga}
Let $g$ be a critical metric for $\A$ on a four-dimensional manifold.
Then
\begin{equation}\label{eq:key}
 2\langle\Ric,\Hess\Scal\rangle-2\Delta|\Ric|^2
 +4|\nabla\Ric|^2+8\langle\Ract\E,\E\rangle=0.
\end{equation}
\end{lemma}

We recall the corresponding classification results.

\begin{prop}[{\cite[Theorem~2.7]{DGAEPS2011}}]\label{prop:DGAEPS}
Let $g$ be a critical metric  for $\A$ on a four-dimensional manifold.
Suppose that
\begin{enumerate}[label=\textup{(\alph*)}]
 \item the scalar curvature is constant;
 \item the squared norm of the Ricci tensor is constant; and
 \item $\langle\Ract\E,\E\rangle\ge0$.
\end{enumerate}
Then $(M,g)$ is either Einstein or locally isometric to a product of
surfaces of constant Gaussian curvatures $c$ and $-c$, where $c\ne0$.
\end{prop}

\begin{prop}[{\cite[Theorem~2.8]{DGAEPS2011}}]\label{th:4.5}
Let $g$ be a critical metric  for $\A$ on a closed, oriented
four-dimensional manifold. If $\langle\Ract\E,\E\rangle\ge0$,
then $(M,g)$ is either Einstein or locally isometric to a product of
surfaces of constant Gaussian curvatures $c$ and $-c$, where $c\ne0$.
\end{prop}
%%%%%%%%%%%%%%%%%%%%
\subsection{The curvature condition and a counterexample}
%%%%%%%%%%%%%%%%
We now examine the curvature condition
$q=\langle\Ract\E,\E\rangle\ge0$ and explain why it cannot be
omitted from Theorem~\ref{thm:main}. 

The curvature action $\Ract$ need not preserve trace-free
symmetric two-tensors. We therefore consider the projected
operator on $S^2_0T^*M$:
\begin{equation*}
 \Rzero=\Pi_0\circ\Ract\big|_{S^2_0T^*M},
 \qquad
 \Pi_0h=h-\frac14(\tr h)g.
\end{equation*}
Since $\langle g,\E\rangle=\tr\E=0$, we have
\begin{equation*}
 q=\langle\Ract\E,\E\rangle
  =\langle\Rzero\E,\E\rangle.
\end{equation*}
Thus the assumption $q\ge0$ requires nonnegativity only when
the quadratic form associated with $\Rzero$ is evaluated at $\E$,
rather than at every element of $S^2_0T^*M$.
For the curvature operator of the second kind,
see \cite{BK78,CGT23,Li24a,NPW23,Nis86}.

To express this condition in terms of sectional curvatures,
choose an orthonormal Ricci eigenbasis $\{e_i\}$ at a point.
Writing $\E(e_a,e_b)=\epsilon_a\delta_{ab}$ and
$K_{ab}=R_{abba}$ for $a\ne b$, we obtain
\begin{equation*}
 q=-2\sum_{a<b}K_{ab}\epsilon_a\epsilon_b.
\end{equation*}
For a product of surfaces with Gaussian curvatures $k$ and $\ell$,
this reduces to
\begin{equation*}
 q=-\frac12(k+\ell)(k-\ell)^2.
\end{equation*}

Consequently, $k+\ell\le0$ at every point of the product
is sufficient for $q\ge0$, even when $k$ and $\ell$ are not constant.

If $k$ and $\ell$ are constant, \eqref{eq:EL-rough} shows that
the product metric is critical for $\A$ if and only if
$k^2=\ell^2$.
In particular, let $c>0$ and let $\Sigma(-c)$ be a complete
noncompact surface of constant Gaussian curvature $-c$
and finite area.
Then the product metric on $S^2(c)\times\Sigma(-c)$ is complete,
non-Einstein, and critical for $\A$, with finite curvature
energy and $q=0$.

The following example shows that the condition $q\ge0$
does not follow from the other hypotheses of
Theorem~\ref{thm:main}.

\begin{example}
\label{rem:sign-counterexample}
Let $a>0$ and equip $M=\mathbb R\times S^3$ with the metric
\begin{equation*}
  g_a=dt^2+f(t)^2g_{S^3},\qquad f(t)=\sqrt{t^2+a^2},
\end{equation*}
where $g_{S^3}$ is the unit round metric. Since
$g_a\geq dt^2+a^2g_{S^3}$, this metric is complete.
 {{Choose a $g_a$-orthonormal frame $\{e_1,e_2,e_3,e_4\}$
with $e_1=\partial_t$ and $e_2,e_3,e_4$ tangent to the slices.
The sectional curvatures of radial and tangential planes are
\begin{equation*}
\begin{aligned}
  K_{\mathrm{rad}}&=K(e_1,e_\alpha)=-\frac{f''}{f}=-\frac{a^2}{(t^2+a^2)^2},\\
  K_{\mathrm{tan}}&=K(e_\alpha,e_\beta)=\frac{1-(f')^2}{f^2}=\frac{a^2}{(t^2+a^2)^2}
\end{aligned}
\end{equation*}
for $2\le\alpha\ne\beta\le4.$
We set
\begin{equation*}
  b(t)=\frac{a^2}{(t^2+a^2)^2}>0.
\end{equation*}
Thus, we have
\begin{equation*}
  \Ric=\operatorname{diag}(-3b,b,b,b),
  \qquad \Scal=0,
  \qquad |\Ric|^2=12b^2.
\end{equation*}}}

To see conformal flatness explicitly, choose the radial coordinate $r$
by $\frac{dr}{r}=\frac{dt}{f(t)}$, with normalization
\begin{equation*}
  r=\frac{t+\sqrt{t^2+a^2}}{2}\in(0,\infty),
  \qquad t=r-\frac{a^2}{4r}.
\end{equation*}
Then
\begin{equation*}
  g_a=\left(1+\frac{a^2}{4r^2}\right)^2
       \left(dr^2+r^2g_{S^3}\right).
\end{equation*}
Hence $g_a$ is conformal to the Euclidean metric on
$\mathbb R^4\setminus\{0\}$ and its Weyl tensor vanishes. This is the
four-dimensional doubled Riemannian Schwarzschild metric~\cite{LeeYeon}.
Since $W=0$ and $\Scal=0$, the local variational identity
$\grad\A
=2\grad\W+\tfrac13\grad\C$
shows that $g_a$ is critical for $\A$ with respect to compactly
supported variations.

With the full tensor norm used in this paper,
\begin{equation*}
  |\Rm|^2=4\left(3K_{\mathrm{rad}}^2
                        +3K_{\mathrm{tan}}^2\right)=24b^2,
  \qquad
  dv_{g_a}=(t^2+a^2)^{3/2}\,dt\,dv_{S^3}.
\end{equation*}

Since $\operatorname{Vol}(S^3,g_{S^3})=2\pi^2$, it follows that
\begin{equation*}
  \A(g_a)
  =48\pi^2a^4\int_{-\infty}^{\infty}
                     (t^2+a^2)^{-5/2}\,dt
  =64\pi^2<\infty.
\end{equation*}
Since $\Scal=0$, we have $\E=\Ric$.
Using the eigenvalues
$(\lambda_1,\lambda_2,\lambda_3,\lambda_4)=(-3b,b,b,b)$,
we obtain
\begin{equation*}
\begin{aligned}
 q&=\langle\Ract\E,\E\rangle
   =-2\sum_{i<j}K_{ij}\lambda_i\lambda_j\\
  &=-2\left(3(-b)(-3b)b+3b\cdot b^2\right)\\
  &=-24b^3<0.
\end{aligned}
\end{equation*}
The Ricci tensor is not parallel, since $|\Ric|^2=12b^2$
is nonconstant. In particular, $g_a$ is neither Einstein nor locally
isometric to a product of surfaces of constant Gaussian curvatures
$c$ and $-c$. Thus completeness, criticality for $\A$, and finite
curvature energy do not suffice for the conclusion of Theorem~\ref{thm:main}
without the curvature assumption $\langle\Ract\E,\E\rangle\ge0$.
\end{example}

%%%%%%%%%%%%%%%%%%%%%%%%%%%%%%%%%%%%%%%%%%%%%%%
\section{Proof of Theorem \ref{thm:main}}\label{s.proof}
{{
The proof uses standard cutoff estimates applied first to the harmonic
scalar curvature equation and then to the Bochner identity for the Ricci
tensor. In the first step, we use the test function
$\eta_s^2\operatorname{Scal}$; in the second, we multiply the identity
in Lemma~\ref{lem:dga} by $\eta_s^2$ and integrate.

Fix $o\in M$. By completeness, there exist smooth cutoff functions
$\eta_s\in C_c^\infty(M)$, $s>1$, satisfying
\begin{equation*}\label{eq:eta_properties}
  0\leq\eta_s\leq1,\qquad
  \eta_s=1\ \text{on }B_s(o),\qquad
  \operatorname{supp}\eta_s\subset B_{2s}(o),\qquad
  |\nabla\eta_s|\leq\frac{C}{s},
\end{equation*}
where $C>0$ is independent of $s$. Such functions can be constructed
by smoothing distance cutoffs. The closed balls are compact by the
Hopf--Rinow theorem. All integrations by parts below involve compactly
supported smooth functions. Finite curvature energy controls the error
terms as $s\to\infty$; no global integrability of
$\nabla\operatorname{Scal}$ or $\nabla\operatorname{Ric}$ is assumed
in advance.

We begin with the following well-known inequalities.

\begin{lemma}[\textup{\cite[p.~74]{BGM}}]\label{eq:finite}
On any Riemannian manifold of dimension $n$,
\begin{equation*}\label{eq:ineq}
\Scal^2 \le n\,|\Ric|^2
\quad \text{and} \quad
|\Ric|^2 \le \frac{n-1}{2}\,|\Rm|^2.
\end{equation*}
In particular, if $\A(g)<\infty$, then $\B(g)<\infty$ and $\C(g)<\infty$.
\end{lemma}

\begin{lemma}\label{lem:3.1}
Let $(M,g)$ be a connected complete four-dimensional Riemannian manifold whose metric is critical for $\A$ and has finite curvature energy.
Then the scalar curvature $\Scal$ is constant on $M$.
\end{lemma}

\begin{proof}
Multiplying \eqref{eq:tau-harmonic} by $\eta_s^2\Scal$ and integrating by parts,
which is justified by the smoothness and compact support of $\eta_s$,
we obtain
\begin{equation*}
\begin{aligned}
0&=\int_M \eta_s^2\Scal\,\Delta\Scal\,\vol\\
 &=-\int_M\langle\nabla(\eta_s^2\Scal),\nabla\Scal\rangle\,\vol\\
 &=-\int_M\big\langle2\eta_s\Scal\,\nabla\eta_s
              +\eta_s^2\nabla\Scal,\nabla\Scal\big\rangle\,\vol\\
 &=-2\int_M\eta_s\Scal\,\langle\nabla\eta_s,\nabla\Scal\rangle\,\vol
   -\int_M\eta_s^2|\nabla\Scal|^2\,\vol.
\end{aligned}
\end{equation*}
Hence
$$
\int_M \eta_s^2|\nabla\Scal|^2\,\vol
= -2\int_M \eta_s\Scal\,\langle\nabla\eta_s,\nabla\Scal\rangle\,\vol,
$$
and therefore
\begin{equation}\label{eq:tri}
\begin{aligned}
\int_M \eta_s^2 |\nabla\Scal|^2\vol
%=& \left| \int_M 2\,\eta_s\,\Scal\,\langle \nabla\eta_s, \nabla\Scal\rangle\vol \right|
\le \int_M 2\,\big|\eta_s\,\Scal\,\langle \nabla\eta_s, \nabla\Scal\rangle\big|\vol.
\end{aligned}
\end{equation}
Applying Cauchy-Schwarz and Young's inequality,
$2|ab|\le \varepsilon a^2+\varepsilon^{-1}b^2$ for $\varepsilon\in(0,1)$, with
$$
a=\eta_s|\nabla\Scal|,
\qquad
b=|\Scal|\,|\nabla\eta_s|,
$$
we have the following estimate:
\begin{equation}\label{eq:young-pointwise}
2\,\big|\eta_s\Scal\,\langle\nabla\eta_s,\nabla\Scal\rangle\big|
\le \varepsilon\,\eta_s^2|\nabla\Scal|^2+\varepsilon^{-1}\Scal^2|\nabla\eta_s|^2
\quad\text{on }M.
\end{equation}
Integrating \eqref{eq:young-pointwise} over $M$ and using \eqref{eq:tri} yields
$$
\int_M \eta_s^2|\nabla\Scal|^2\,\vol
\le \varepsilon\int_M \eta_s^2|\nabla\Scal|^2\,\vol+\varepsilon^{-1}\int_M \Scal^2|\nabla\eta_s|^2\,\vol,
$$
and hence
\begin{equation}\label{eq:etaR-grad-tau-estimate}
\int_M \eta_s^2|\nabla\Scal|^2\,\vol
\le C_1\int_M \Scal^2|\nabla\eta_s|^2\,\vol,
\end{equation}
for some constant $C_1>0$ independent of $s$. Using $|\nabla\eta_s|\le C/s$ and $\operatorname{supp}(\nabla\eta_s)\subset B_{2s}(o)\setminus B_s(o)$, we obtain
$$
\int_M \Scal^2|\nabla\eta_s|^2\,\vol
\le \frac{C^2}{s^2}\int_{B_{2s}(o)}\Scal^2\,\vol
\le \frac{C^2}{s^2}\int_M \Scal^2\,\vol.
$$
By Lemma~\ref{eq:finite}, the scalar curvature $\Scal$ has finite $L^2$-norm, and since $C^2/s^2\rightarrow0$ as $s\rightarrow\infty$, it follows that
$$
\lim_{s\rightarrow\infty}\int_M \Scal^2|\nabla\eta_s|^2\,\vol=0.
$$
Therefore \eqref{eq:etaR-grad-tau-estimate} implies
$$
\lim_{s\rightarrow\infty}\int_M \eta_s^2|\nabla\Scal|^2\,\vol=0.
$$
For any fixed $R>0$ and $s\ge R$, we have $\eta_s=1$ on $B_R(o)$.
Consequently,
\begin{equation*}
 0\le\int_{B_R(o)}|\nabla\Scal|^2\,\vol
 \le\int_M\eta_s^2|\nabla\Scal|^2\,\vol\longrightarrow0.
\end{equation*}
Since $R$ is arbitrary, we have
{\mbox{$\nabla\Scal\equiv0$}} on $M$. Thus $\Scal$ is constant on $M$.
\end{proof}

\begin{lemma}\label{lem:3.2}
Let $(M,g)$ be a connected complete four-dimensional Riemannian manifold whose metric is critical for $\A$ and has finite curvature energy.
Assume that $\langle\Ract\E,\E\rangle\ge0$ holds on $M$. Then the Ricci tensor is parallel, i.e.,
$ \nabla\operatorname{Ric}\equiv0.$
In particular, $|\operatorname{Ric}|^2$ is constant on $M$. 
\end{lemma}

\begin{proof}
Multiplying \eqref{eq:key} by $\eta_s^2$ and integrating over $M$, we obtain

\begin{equation}\label{eq:3.1}
\begin{aligned}
0={}&\int_M\eta_s^2\bigl(2\Ric^{ij}\nabla_i\nabla_j\Scal-2\Delta|\Ric|^2\bigr)\,\vol\\
   &+\int_M\eta_s^2\bigl(4|\nabla\Ric|^2+8\langle\Ract\E,\E\rangle\bigr)\,\vol.
\end{aligned}
\end{equation}

By Lemma \ref{lem:3.1}, the scalar curvature $\Scal$ is constant on $M$ and so the Hessian term in \eqref{eq:3.1} vanishes. Integrating the Laplacian term by parts, we obtain
\begin{equation*}
-2\!\int_M\!\eta_s^2\Delta|\Ric|^2\,\vol
=4\!\int_M\!\eta_s\, \langle\nabla|\Ric|^2,\nabla\eta_s\rangle\,\vol .
\end{equation*}
Therefore \eqref{eq:3.1} becomes
\begin{equation*}
\begin{aligned}
0={}&4\int_M\eta_s\langle\nabla|\Ric|^2,\nabla\eta_s\rangle\,\vol\\
   &+4\int_M\eta_s^2|\nabla\Ric|^2\,\vol
    +8\int_M\eta_s^2\langle\Ract\E,\E\rangle\,\vol.
\end{aligned}
\end{equation*}

Then we have the following:
\begin{equation*}
\begin{aligned}
4\int_M\eta_s^2|\nabla\Ric|^2\,\vol &=-4\int_M\eta_s\langle\nabla|\Ric|^2,\nabla\eta_s\rangle\,\vol\\
 &\quad-8\int_M\eta_s^2\langle\Ract\E,\E\rangle\,\vol.
\end{aligned}
\end{equation*}

Since $\langle\Ract\E,\E\rangle\ge 0$,  we obtain the following inequality:
\begin{equation}\label{eq:star}
\begin{aligned}
4\int_M\eta_s^2|\nabla\Ric|^2\,\vol
 &\le-4\int_M\eta_s\langle\nabla|\Ric|^2,\nabla\eta_s\rangle\,\vol\\
 &\le4\int_M\eta_s\big|\langle\nabla|\Ric|^2,\nabla\eta_s\rangle\big|\,\vol.
\end{aligned}
\end{equation}
Using $\nabla|\Ric|^{2}=2\langle\nabla\Ric,\Ric\rangle$, we estimate the right-hand side of \eqref{eq:star}
by the Cauchy-Schwarz inequality and Young's inequality with $\varepsilon\in(0,4)$ as follows:
\begin{equation*}
\begin{aligned}
&4\int_M\eta_s\big|\langle\nabla|\Ric|^2,\nabla\eta_s\rangle\big|\,\vol\\
 &\quad=4\int_M\eta_s\big|\langle2\langle\nabla\Ric,\Ric\rangle,\nabla\eta_s\rangle\big|\,\vol\\
 &\quad\le8\int_M\eta_s|\Ric|\,|\nabla\Ric|\,|\nabla\eta_s|\,\vol\\
 &\quad\le\varepsilon\int_M\eta_s^2|\nabla\Ric|^2\,\vol
          +\frac{16}{\varepsilon}\int_M|\Ric|^2|\nabla\eta_s|^2\,\vol.
\end{aligned}
\end{equation*}

Then we have 
\begin{equation*}
    \int_M \eta_s^{2}|\nabla\Ric|^{2}\vol
\le
\frac{16}{\varepsilon(4-\varepsilon)}\int_M |\Ric|^{2}|\nabla\eta_s|^{2}\vol.
\end{equation*}
Thus
\begin{equation}\label{eq:rho}
 \int_M \eta_s^2|\nabla\Ric|^2\,dv_g
 \le C_2\int_M|\Ric|^2|\nabla\eta_s|^2\,dv_g
\end{equation}
for some constant $C_2>0$ independent of $s$.
Using $|\nabla\eta_s|\le C/s$ and the fact that
$\eta_s$ is supported in $B_{2s}(o)$, we obtain
$$
\int_M |\Ric|^2|\nabla\eta_s|^2\,\vol
 \le \frac{C^2}{s^2} \int_{B_{2s}(o)} |\Ric|^2\,\vol
 \le \frac{C^2}{s^2} \int_M |\Ric|^2\,\vol.
$$
By Lemma \ref{eq:finite}, we see that $\int_M|\Ric|^2\vol<\infty$ and  the factor
$C^2/s^2$ tends to zero as $s\rightarrow\infty$. Thus
$$
\lim_{s\rightarrow\infty}
\int_M |\Ric|^2|\nabla\eta_s|^2\,\vol
= 0.
$$
From \eqref{eq:rho} we therefore get
$$
\lim_{s\rightarrow\infty}
\int_M \eta_s^{2}|\nabla\Ric|^{2}\vol
= 0.
$$
For each fixed $R>0$, letting $s\to\infty$ with $s\ge R$ gives
\begin{equation*}
 0\le\int_{B_R(o)}|\nabla\Ric|^2\,\vol
 \le\int_M\eta_s^2|\nabla\Ric|^2\,\vol\longrightarrow0.
\end{equation*}
Since $R>0$ is arbitrary, it follows that
$\nabla\operatorname{Ric}\equiv0$ on $M$. Consequently,
\(
  \nabla|\operatorname{Ric}|^2
  =2\langle\nabla\operatorname{Ric},\operatorname{Ric}\rangle
  \equiv0,
\)
so $|\operatorname{Ric}|^2$ is constant on $M$.
\end{proof}
With Lemmas~\ref{lem:3.1} and~\ref{lem:3.2} in hand, we now complete the proof of Theorem~\ref{thm:main}.

\begin{proof}[Proof of Theorem~\ref{thm:main}]
Let $(M,g)$ satisfy the hypotheses of Theorem~\ref{thm:main}.
By Lemma~\ref{lem:3.1} we see that $M$ has constant scalar curvature, and by Lemma~\ref{lem:3.2} we have that $|\Ric|^2$ is constant on $M$.
Hence $(M,g)$ satisfies all the assumptions of Proposition~\ref{prop:DGAEPS}. This completes the proof.
\end{proof}

%%%%%%%%%%%%%%%%%%%%%%%%%%%%%%%%%%%%%%%%%%%%%%%
%Section 4.
%%%%%%%%%%%%%%%%%%%%%%%%%%%%%%%%%%%%%
\section{The non-four-dimensional cases }\label{s:lowdim}

In dimension $n=2$, the curvature tensor is determined by the Gaussian
curvature, and $|\Rm|^2=\Scal^2$. Thus criticality for $\A$ is equivalent
to criticality for $\C$. 
{{The flatness of complete critical metrics for $\A$ therefore
follows from \cite[Theorem~1.1]{CMM}.}}

In dimension $n=3$, the identity
$|\Rm|^2=4|\Ric|^2-\Scal^2$ gives
\begin{equation*}
 \A=4\mathfrak F_{-1/4}^2,\qquad
 \mathfrak F_t^2(g)=\int_M|\Ric|^2\,\vol
                   +t\int_M\Scal^2\,\vol.
\end{equation*}

{{Consequently, criticality for $\A$ is equivalent to criticality
for $\mathcal{F}^{2}_{-1/4}$.
Since $-\frac14>-\frac13$,
\cite[Theorem~1.2]{CMM} implies that every complete
three-dimensional critical metric for $\A$ is flat.}}
For comparison, Anderson~\cite[Theorem~0.1]{And} proved flatness of
complete three-dimensional critical metrics for $\B$ with nonnegative
scalar curvature, while Catino~{\cite[Theorem~1.2]{Catino}} proved that
complete three-dimensional critical metrics for $\C$ with nonnegative
scalar curvature are scalar flat. 

\begin{remark}\label{rem:compact-vs-complete}
Let $(M,g)$ be a closed $n$-dimensional Riemannian manifold, $n\ne4$,
whose metric is critical for $\A$. Then $g$ satisfies
$$\Delta \Scal = \frac{n-4}{4}|\Rm|^2.$$
Integrating over $M$ and using Stokes' theorem, we obtain
$$0=\int_M \Delta\Scal\,dv_g=\frac{n-4}{4}\int_M |\Rm|^2\,dv_g.$$
Since $n\neq 4$, this implies {\mbox{$\int_M |\Rm|^2\,dv_g=0$}}, hence $|\Rm|\equiv0$ and $(M,g)$ is flat.%

In contrast, if $M$ is complete and noncompact, the identity
{\mbox{$\int_M \Delta\Scal\,dv_g=0$}} is not automatic. In dimensions $n = 2$ and $n = 3$, every complete
critical metric  for $\A$ is flat without any additional assumptions. In dimensions {\mbox{$n>4$}}, flatness follows, for example, under the
additional assumption $|\nabla\Scal|\in L^1(M)$. 
This follows from the trace equation by a standard cutoff argument. In dimension four, the trace equation instead reduces to $\Delta\Scal=0$.
\end{remark}

\end{document}